\documentclass[11pt,leqeq]{article}
\usepackage{amssymb,amsthm}
\usepackage{amsmath}
\usepackage{amscd}
\usepackage[dvips]{graphicx}
\newtheorem{thm}{Theorem}%[section] (If you want theorem numbered
\newtheorem{lema}[thm]{Lemma}%               with section number.  Same
\newtheorem{prop}[thm]{Proposition} %--> \begin\end{theorem,lemma,...}

\newtheorem{defi}[thm]{Definition}
\newtheorem{exa}[thm]{Example}

\newcommand{\en}{{\rm{End}}}
\newcommand{\des}{\displaystyle}

\newcommand{\sym}{{\rm{Sym}}}

\begin{document}
\date{}

\title{On the $q$-meromorphic Weyl algebra}
\author{Rafael D\'\i az and Eddy Pariguan} \maketitle
\begin{abstract}
We introduce a $q$-analogue $MW_q$ for the meromorphic Weyl algebra, and
study the normalization problem and the symmetric powers $\sym^n(MW_q)$ for such
algebra from a combinatorial viewpoint.
\end{abstract}

\section{Introduction}
Pioneered by Euler, Jacobi, and Jackson among others, the results
and applications of $q$-calculus \cite{RA, Ch} have grown both in
depth and scope, touching by now most branches of mathematics,
including partition theory \cite{AND}, combinatorics
\cite{z1, z2}, number theory
\cite{Kim}, hypergeometric functions \cite{RA}, quantum groups
\cite{kass}, knot theory \cite{Tu}, $q$-probabilities \cite{kuper}, Gaussian
$q$-measure \cite{CTT}, Feynman $q$-integrals \cite{ED3, ED1},
homological algebra \cite{A, Kap}, and category theory
\cite{cas1}. Our goal in this
work is to bring yet another mathematical object into the field of
$q$-calculus, namely, we provide a $q$-analogue for the
meromorphic Weyl algebra $MW$ introduced in \cite{DP2}.
Roughly speaking $MW$ is the algebra generated by $x^{-1}$ and the
derivative $\partial$. The $q$-analogue $MW_q$ of the meromorphic
Weyl algebra is essentially the algebra generated by $x^{-1}$ and
the $q$-derivative $\partial_q$. We focus on the normal
polynomials for $MW_q$ which arise in the problem of writing
arbitrary monomials in $MW_q$ as linear combination of monomials
written in normal form; we provide both explicit formulae and
a combinatorial interpretation for the normal polynomials. We also
study the symmetric powers of $MW_q$ using
the methodology developed in \cite{DP2} and further applied in \cite{RDEP, DR}.\\

Let us say a few words on  $q$-combinatorics. As explained by
Zeilberger in \cite{z2}  a combinatorial interpretation for a
sequence $n_0, n_1, n_2,....$ of non-negative integers, is a
sequence of finite sets $x_0, x_1, x_2,...$ such that $|x_k|=n_k$
for $k \in
\mathbb{N}$. Each sequence of non-negative integers admits a wide
variety of combinatorial interpretations; the art of combinatorics
consists in finding patterns that yield, systematically,
combinatorial interpretations for  families of sequences of
non-negative integers.\\

The field of $q$-combinatorics provides another approach for the
study of natural numbers by combinatorial methods. Let
$\mathbb{N}[q]$ be the semi-ring of polynomials in the variable
$q$ with coefficients in $\mathbb{N}.$ Instead of working with
sequences of finite sets the main object of study in
$q$-combinatorics are sequences $(x_0,\omega_0), (x_1,
\omega_1), (x_2, \omega_2), ...$ of pairs $(x,\omega)$ where $x$
is a finite set and $\omega:x \longrightarrow \mathbb{N}[q]$ is an
arbitrary map. The cardinality of such a pair $(x,\omega)$ is
defined to be
$$|x,\omega|=
\sum_{i \in x}\omega(i) \in \mathbb{N}[q].$$

Notice that the cardinality $|x,\omega|$ of the pair $(x,\omega)$ is not an
integer, but rather a polynomial in the variable $q$ with
non-negative integer coefficients. We say that a sequence of pairs
$(x_0,\omega_0), (x_1, \omega_1), (x_2, \omega_2), \cdots$
provides a combinatorial interpretation for a sequence of
non-negative integers $n_0, n_1, n_2,\cdots$ if $|x_k,
\omega_k|(1)= n_k$ for $k \in \mathbb{N}$, where $|x_k,
\omega_k|(1)$ is the evaluation of the polynomial $|x_k,
\omega_k|$ at $1$. Of course the additional value of
$q$-combinatorics comes from the fact that it is suited to handle
not just sequences in $\mathbb{N}$, but more generally sequences
in $\mathbb{N}[q]$. We say that a sequence $(x_0,\omega_0), (x_1,
\omega_1), (x_2, \omega_2), \cdots$ provides a combinatorial
interpretation for a sequence of polynomials $p_1, p_2,
p_3,\cdots$ in $\mathbb{N}[q]$ if $|x_k,\omega_k|=p_k$ for $k \in
\mathbb{N}.$ One of the most prominent examples is the
$q$-combinatorial interpretation for the $q$-analogues $[n]! \in
\mathbb{N}[q]$ of the factorial numbers $n!$  given by
$$[n]!=\prod_{k=1}^n[k]  \mbox{ \ \ where  \ \ }
[k]=1+\cdots+q^{k-1}.$$

Consider the pair $(S_n,i_n)$ where $S_n$
is the set of permutations of $[[1,n]]=\{1,2,\cdots,n \}$ and
$i_n: S_n \longrightarrow \mathbb{N}[q]$ is the map given by
$i_n(\sigma)=q^{|I_n(\sigma)|}$ where
$$I_n(\sigma)=\{(i,j) \ \ | \ \ 1\leq i < j\leq n \mbox{\  \ and \ }\sigma(i) > \sigma(j)    \} .$$ An inductive argument \cite{AND, ED1} shows that
$|S_n, i_n|=[n]!,$ therefore the sequence $(S_n, i_n)$
provides a combinatorial interpretation for $[n]!$.\\

The rest of this work is organized as follows. In Section $2$ we
summarize some facts on the meromorphic Weyl algebra; we do not
include proofs since all the stated results are consequences,
setting $q=1$, of the corresponding $q$-analogue results proved in
the subsequent sections. The main results of this work are given
in Sections $3$ and $4$ where we introduce $MW_q$ the $q$-analogue
of the meromorphic Weyl algebra, discuss its basic properties,
provide a couple of representations for it, study the normal
polynomials that arise in the process of writing monomials in $MW_q$ in
normal form, and begin the study of the symmetric powers $\sym^n(MW_q)$ of the
$q$-meromorphic Weyl algebra.

\section{The meromorphic Weyl algebra}

The Weyl algebra is the associative algebra over the
field of complex numbers $\mathbb{C}$  given by
$$W=\mathbb{C} \langle x,y \rangle/
\langle yx-xy- 1\rangle$$ where $\mathbb{C}
\langle x,y \rangle$ is the free associative algebra over $\mathbb{C}$ generated by
formal variables $x$ and $y$, and $\langle yx-xy- 1\rangle$ is the
ideal generated by $yx-xy- 1.$ The Weyl algebra comes with a
natural representation $$\rho: W \longrightarrow
End(\mathbb{C}[x]),$$ where  $\mathbb{C}[x]$ is the vector space
of polynomials in the variable $x$ and $End(\mathbb{C}[x])$ is the
algebra of endomorphisms of  $\mathbb{C}[x]$, which explain why it
appears so often in many branches of mathematics and physics. The
map $\rho$ is given on the generators of $W$ by
$$\rho(x)f=xf \mbox{ \ \  and \ \ }
\rho(y)f=\frac{\partial
f}{\partial x}.$$

Notice that in the definition above the letter
$x$ on the left-hand side is a non-commutative variable, while  on
the right-hand side the letter $x$ denotes the generator of
$\mathbb{C}[x]$. This sort of abuse of notation is common in the
literature and we hope it causes no confusion.\\

The meromorphic Weyl algebra $MW$ is the associative algebra over
$\mathbb{C}$ given by
$$MW=\mathbb{C}
\langle x,y
\rangle/
\langle yx-xy- x^{2}\rangle.$$

\noindent $MW$ comes with a natural representation $\rho$ which justifies
its name. Let $C^{\infty}(\mathbb{R}^*)$ be the space of smooth complex valued
functions on the punctured real line
$\mathbb{R}^*=\mathbb{R}\setminus \{0\}$. The representation
$$\rho:MW
\longrightarrow
\en(C^{\infty}(\mathbb{R}^*))$$ is defined by letting the generators of
$MW$ act on $f\in C^{\infty}(\mathbb{R}^*)$ as follows:
$$\rho(x)f=x^{-1}f \mbox{ \ \  and \ \ }
\rho(y)f=-\frac{\partial
f}{\partial x}.$$

An integral analogue of the Weyl algebra is obtained by
considering the operators $l(x)$ and $l(y)$ acting on $f \in
C^{\infty}(\mathbb{R})$ as follows:
$$l(x)f =xf
\mbox{ \ \ and \ \ } l(y)f=\int_0^x f(t)dt.$$ It is not hard to see that $l$ extends
naturally to yield a representation
$$l:\mathbb{C} \langle x,y \rangle/
\langle yx-xy + y^{2}\rangle \longrightarrow  \en(C^{\infty}(\mathbb{R}))$$
of the  algebra $$\mathbb{C} \langle x,y \rangle/
\langle yx-xy + y^{2}\rangle,$$ which is isomorphic to the meromorphic Weyl algebra via
the isomorphism $$t: MW \longrightarrow \mathbb{C} \langle x,y
\rangle/ \langle yx-xy + y^{2}\rangle$$ given on generators by
$t(x)=y$ and $t(y)=x$.   Thus  the map $\iota:MW
\longrightarrow \en(C^{\infty}(\mathbb{R}))$ given on generators by
$$\iota(x)f=\int_0^{\infty} f(t)dt \mbox{ \ \  and \ \ } \iota(y)f=x
f$$ defines a  representation of the meromorphic Weyl algebra. \\

We will use the following notation. For $A=(A_1,\cdots,A_n)\in
(\mathbb{N}^{2})^{n}$ where $A_i=(a_i,b_i)$, we set $a=(a_1,...,a_n)$,
$b=(b_1,...,b_n)$, and
$|A|=(|a|,|b|)=(a_1+\cdots+a_n, b_1+\cdots+b_n).$\\

The normal coordinates $ N(A,k)$ of the monomial
$\prod_{i=1}^nx^{a_i}y^{b_i}\in MW$ are given by
$$\prod_{i=1}^nx^{a_i}y^{b_i}=
\sum_{k=0}^{|b|} N(A,k)x^{|a|+k}y^{|b|-k}.$$ For $k>|b|$ we set $N(A,k)=0$. \\

Given vector $a=(a_1,\cdots,a_n)$ then for $i \in [[1, n-1]]$ we
let $a_{>i}$ be the vector $(a_{i+1},\cdots,a_{n})$. The
increasing factorial \cite{GCRota} is given by $$n^{(k)}=
n(n+1)(n+2)\cdots(n+k-1)$$ for $n \in \mathbb{N}$ and $k \geq 1$
an integer. In the statement of the Theorem \ref{emnc} the
notation $p
\vdash k$ means that $p$ is a vector $(p_1,\cdots,p_{n-1}) \in \mathbb{N}^{n-1}$ such
that $|p|={\sum_{i=1}^{n-1}p_i=k}$.

\begin{thm}\label{emnc}{\em
For $(A,k) \in (\mathbb{N}^{2})^{n} \times \mathbb{N}$  the
following identity holds
$$N(A,k)= {\des\sum_{p\vdash k}{b\choose p}
\prod_{i=1}^{n-1}(|a_{>i}|+|p_{>i}|)^{(p_i)}},$$
where $${b\choose p}=\prod_{i=1}^{n-1}{b_i \choose p_i} .$$}
\end{thm}

The numbers $N(A,k)$ have a nice combinatorial meaning. Let
$E_1,\dots,E_n,F_1,\dots,F_n$ be disjoint sets such that
$|F_i|=a_i$, $|E_i|=b_i$ for $i\in [[1,n]]$, and set $E=\sqcup
E_i$, $F=\sqcup F_i$. Let $M_k$ be the set whose elements are maps
$f:F\longrightarrow \{\mbox{ subsets of }E \ \}$ such that:
\begin{itemize}
\item $f(x)\cap f(y)=\emptyset$ for $x,y \in F;$
\item if $y\in f(x), \ x\in F_i, \  y\in E_j,$ then   $j<i;$
\item $\sum_{a\in F} |f(a)|=k.$
\end{itemize}

The sets $M_k$ provide a combinatorial interpretation for the
numbers $N(A,k)$, that is $$|M_k|=N(A,k).$$ Figure
$\ref{fig:mero}$ illustrates the combinatorial interpretation for
$N(((2,3), (3,3), (3,4)),6):$ it shows an example of
a map contributing to $N(((2,3), (3,3), (3,4)),6)$.
\begin{figure}[h]
\begin{center}
\includegraphics[width=2.5in]{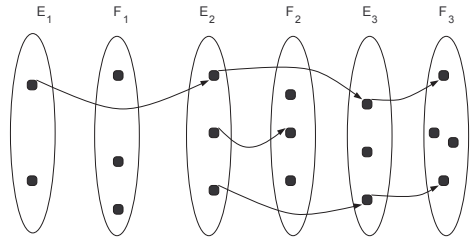}
\caption{Combinatorial interpretation of $N(((2,3), (3,3), (3,4)), 6)$.
\label{fig:mero}}
\end{center}
\end{figure}

Applying Theorem $\ref{emnc},$ specialized in the representation
$\rho$, to $x^{-t} \in C^{\infty}(\mathbb{R}^*)$  we obtain for $(a,b,t)\in
\mathbb{N}^{n}\times
\mathbb{N}^{n} \times \mathbb{N}_+$ the following identity:
$$\prod_{i=1}^{n} (t+|a_{>i}|+|b_{>i}|)^{(b_i)}=\sum_{p\vdash k}{b\choose p}
\prod_{i=1}^{n-1}(|a_{>i}|+ |p_{>i}|)^{(p_i)}t^{(|b|-k)}.$$

This identity is thus an easy corollary of Theorem
\ref{emnc}; however guessing or even proving it directly could
be a bit of a pain. Applying Theorem $\ref{emnc}$, specialized in
the representation $\iota$, to $x^{t}$  we get another quite
intriguing identity:
$$\frac{1}{\prod_{i=1}^{n}(t + |a_{>i}| + |b_{\geq i}| + 1)^{(a_i)} }= \sum_{p\vdash k}{b\choose p}
\prod_{i=1}^{n-1}\frac{(|a_{>i}|+
|p_{>i}|)^{(p_i)}}{(t+|b|-k+1)^{(|a|+k)}}.$$

A fundamental yet not fully appreciated fact in algebra is that
one can associate with each associative algebra $A$ a family of
associative algebras $Sym^n(A)$ indexed by the natural numbers $n
\in \mathbb{N}$. Formally, let $\mathbb{C}$-$alg$ be the category of
associative complex algebras. For $n\geq 1$ consider
$$\sym^n:\mathbb{C}\mbox{-}alg \longrightarrow
\mathbb{C}\mbox{-}alg$$  the functor sending an algebra $A$ into its $n$-th
symmetric power given by
$$\sym^n(A)=A^{\otimes n}/\langle a_{1}\otimes\cdot\cdot\cdot\otimes
a_{n}-a_{\sigma^{-1}(1)}\otimes\cdot\cdot\cdot\otimes
a_{\sigma^{-1}(n)}\ | \ a_{i}\in A,\ \sigma\in S_n\rangle.$$

Given $a_1 \otimes ... \otimes a_n \in A^{\otimes n}$ we denote by $\overline{a_1 \otimes ... \otimes a_n}$ the corresponding element in $\sym^n(A).$ The rule  for the product of $m$ elements in $\sym^n(A)$, see \cite{DP2},
is given as follows: let $a_{ij} \in A$ for $(i,j)
\in [[1,m]]\times[[1,n]]$, then we have that
$$n!^{m-1}\prod_{{i=1}}^{m}
\overline{\bigotimes_{{j=1}}^{n}a_{ij}}=\sum_{\sigma\in
\{1\}\times S_n^{m-1}}\overline{\bigotimes_{{j=1}}^{n}
\prod_{{i=1}}^{m}a_{i\sigma_{{i}}^{-1}(j)}},$$
where $1$ denotes the identity permutation.\\

To our knowledge the symmetric powers have been fully studied only
for a few algebras: for the algebra of polynomials
whose symmetric powers may be identified with the algebra of
symmetric polynomials; and for the algebra of matrices whose
symmetric powers may be identified with the so called Schur
algebras \cite{DP2}. The symmetric powers of the Weyl algebra and
its $q$-analogues are studied in \cite{DP2, RDEP}, the symmetric
powers of the linear Boolean algebras are studied in
\cite{DR}.\\

Let $\sym^{n}(MW)$ be the $n$-symmetric
power of the meromorphic Weyl algebra. An explicit formulae for the
product of $m$ elements in $\sym^{n}(MW)$ is provided next. We denote the element
$$\overline{x^{a_1}y^{b_1}\otimes ... \otimes x^{a_n}y^{b_n}}\in \sym^{n}(MW)
\mbox{\  \ by \ \ } \overline{\prod_{j=1}^{n} x_j^{a_{j}}y_j^{b_{j}}}. $$

\begin{thm}{\em
For each map $(a,b):[[1,m]]\!\times\! [[1,n]]\longrightarrow \mathbb{N}^{2}$
the following identity holds in $\sym^{n}(MW):$
\begin{equation*}
{\des (n!)^{m-1}\prod_{i=1}^{m}\overline{\prod_{j=1}^{n}
x_j^{a_{ij}}y_j^{b_{ij}}}=\sum_{\sigma,k,p}\left(\prod_{l=1}^{m-1}\prod_{j=1}^{n} {
b_j^{\sigma} \choose p^{j} }
(|{(a_{j}^{\sigma})}_{>l}|+|p^{j}_{>l}|)^{(p_l^{j})}\right)
\overline{\prod_{j=1}^{n} x_j^{|a_j^{\sigma}| + k_j}y_j^{|b_j^{\sigma}|-k_j}}.}
\end{equation*}

In the formula above we are using the following conventions:
$\sigma\in {\{1 \} \times S_n^{m-1}}$, $k\in\mathbb{N}^{n}$ is such that
$k_j \leq |b_j^{\sigma}|$, $p=(p^1,...,p^{n})\in (\mathbb{N}^{m-1})^{n}$, $p^j=( p^j_1,...,p_{m-1}^j )$,
$a_j^{\sigma}=(a_{1\sigma_1^{-1}(j)},..., a_{m\sigma_m^{-1}(j)}),$ and
$b_j^{\sigma}=(b_{1\sigma_1^{-1}(j)},..., b_{m\sigma_m^{-1}(j)})$

}
\end{thm}

The next  example shows the high computational power required
to compute even the simplest products in the symmetric powers of
the meromorphic Weyl algebra.
\begin{exa}{\em For
$n=2, m=2$ we have
$$2(x_1y_1^2x_2^2y_2^2)(x_1^2y_1x_2y_2^2)=
x_1^3y_1^4x_2^3y_2^4+6x_1^3y_1^4x_2^4y_2^3+8x_1^3y_1^4x_2^5y_2^2+
+8x_1^4y_1^3x_2^4y_2^3$$
$$+20x_1^4y_1^3x_2^5y_2^2+6x_1^5y_1^2x_2^3y_2^4+12x_1^5y_1^2x_2^5y_2^2
+x_1^3y_1^4x_2^4y_2^4+2x_1^3y_1^4x_2^5y_2^3+6x_1^3y_1^4x_2^6y_2^2$$
$$
+2x_1^4y_1^3x_2^4y_2^4+4x_1^4y_1^3x_2^5y_2^3+12x_1^4y_1^3x_2^6y_2^3+6x_1^5y_1^2x_2^4y_2^4
+12x_1^5y_1^2x_2^5y_2^3+36x_1^5y_1^2x_2^6y_2^2 .$$}
\end{exa}

\section{The $q$-meromorphic Weyl algebra}

In this section we introduce the $q$-meromorphic Weyl algebra and
discuss some of its basic properties. Let us first review a few
basic notions of $q$-calculus; the interested reader may consult
\cite{Ch, So, CTT} for further information.
Let $M(\mathbb{R}^*)$ be the space of complex value functions
defined on the punctured real line $\mathbb{R} \setminus \{0 \}$ and fix a positive real number
$0< q <1$. The $q$-derivative $$\partial_q: M(\mathbb{R}^*)
\longrightarrow M(\mathbb{R}^*)$$ is given by
$$\displaystyle{\partial_q f=\frac{I_q f-
f}{(q-1)x}},$$ where $I_q f(x)=f(qx)$ for $x\in\mathbb{R}^*.$

\begin{defi}{\em
The $q$-meromorphic Weyl is the algebra given by
$$MW_q=\mathbb{C}\langle x,y\rangle[q]/\langle
yx-qxy-x^2\rangle,$$ where $\mathbb{C}\langle x,y\rangle[q]$ is
the free associative algebra generated by the non-commuting
variables $x,y$ and the commutative variable $q$. }
\end{defi}

Notice that in the definition above $q$ is used as a formal variable rather than a number.
It should always be clear from the context whether we are using $q$ as a formal variable or as a number.
Next result explains how the algebra $MW_q$ arises  in
$q$-calculus. For our next result we make use of the $q$-Leibnitz rule

$$\partial_q(fg)= f \partial_q g + I_q g \partial_q f.$$

\begin{thm}\label{rmw}{\em
a
The map $\rho:MW_q \longrightarrow \en(M(\mathbb{R}^*))$ given on
generators by
$$\rho(x)f=x^{-1}f, \ \ \
\rho(y)f=-q^{-1}\partial_{q^{-1}}f, \mbox{\ \  and \ } \rho(q)f=qf$$ for $f\in
M(\mathbb{R}^*)$ defines a representation of $MW_q$.}
\end{thm}

\begin{proof}
We must prove that $$\rho(y)\rho(x)f= q
\rho(x)\rho(y)f+\rho(x^2)f.$$ Since $\partial_{q^{-1}}x^{-1}=-qx^{-2}$ we find
that
\begin{eqnarray*}
\rho(y)\rho(x)f&=&\rho(y)(x^{-1}f)=-q^{-1}
\partial_{q^{-1}}(x^{-1}f)\\
\mbox{} &=&-q^{-1}(q^{-1}x)^{-1}
\partial_{q^{-1}}f - q^{-1}f\partial_{q^{-1}}(x^{-1})\\
\mbox{} &=&-x^{-1}\partial_{q^{-1}}f +x^2f\\
\mbox{} &=&q\rho(x)\rho(y)f + \rho(x^2)f.
\end{eqnarray*}
\end{proof}

Recall \cite{Ch} that the Jackson integral of a map $f:
\mathbb{R} \longrightarrow \mathbb{R}$ is given by
$$\int_0^x f(t)d_qt = (1-q)x\sum_{n=0}^{\infty}q^nf(q^n x).$$

A non-fully exploited feature of the Jackson integral is that it
satisfies a twisted form of the Rota-Baxter identity \cite{cas1,
DPz, GCRota}; indeed one can show that
$$\left(\int_0^x f(s)d_qs \right)\left(\int_0^x g(t)d_qt \right)=
\int_0^x \left(\int_0^t f(s)d_qs \right)g(t)d_qt +
\int_0^x f(t)\left(\int_0^{qt} g(s)d_qs \right)d_qt.$$

It is not hard to check that the Jackson integral is a right inverse operator  for the
$q$-derivative, that is $$\partial_q\int_0^{x} f(t)d_qt=f(x).$$

From the $q$-Leibnitz rule and the fundamental theorem of $q$-calculus
one obtains the $q$-integration by parts formula
$$\int_0^x I_qf\partial_q gd_qt=f(x)g(x) -f(0)t(0)-
\int_0^x g\partial_qf d_qt.$$

In particular setting $$f(x)=x \mbox{\ \ and \ \ } g(x)= \int_0^x
f(t) d_qt$$ we obtain the relation $$x\int_0^x f d_qt = q \int_0^x
tf d_qt +
\int_0^x\int_0^t f d_qs d_qt.$$

Let $I(\mathbb{R})$ be a space of functions on the real line
closed under Jackson integration and under multiplication by
polynomial functions. The previous considerations give  the
following result.

\begin{thm}\label{rmw}{\em   The map
 $$\iota:MW_q \longrightarrow
\en(I(\mathbb{R}))$$ given on generators by
$$\iota(x)f=\int_0^x f d_qt,  \ \ \
\iota(y)f=xf,  \mbox{\ \  and \ } \iota(q)f=qf,$$ for $f\in I(\mathbb{R})$ defines a representation of $MW_q$.}
\end{thm}

We order the generators of $MW_q$ as  $q<x<y$. A monomial in
$MW_q$ of the form $q^a x^b y^c$ is said to be in normal form. One
can show that the set monomials in normal form is a basis for
$MW_q$. Recall from the introduction that we are writing
$[n]=1+...+q^{n-1}$ for an integer $n \geq 1$.

\begin{lema}\label{p1}{\em For $n \geq 1$ the identity $yx^n=q^nx^ny + [n]x^{n+1}$ holds in $MW_q$.}
\end{lema}
\begin{proof}
For $n=1$ we get $yx=qxy + x^2$. By induction we have that
$$yx^{n+1}=yx^{n}x=(q^nx^ny + [n] x^{n+1})x=q^nx^n(yx) + [n] x^{n+1}x=q^{n+1}x^{n+1}y + [n+1] x^{n+2}.$$
\end{proof}

\begin{defi}\label{nc}{\em Let $(a,b) \in \mathbb{N}$  and $0 \leq k \leq
a$. The normal coordinates $c(a,b,k)$ are the elements of  $\mathbb{N}[q]$ given
by the following identity in $MW_q$:
$$y^ax^b=\sum_{k=0}^a c(a,b,k)x^{b+k}y^{a-k}.$$ For $k > a$ we set $c(a,b,k)=0$. Notice that by definition $c(0,b,k)= \delta_{0,k}$ where
$\delta$ is Kronecker's delta function.}
\end{defi}

\begin{prop}\label{ayuda}{\em The following identities hold in $MW_q$:
\begin{enumerate}
\item{$c(a+1,b,k)=c(a,b,k)q^{b+k} + c(a,b,k-1)[b+k-1]$ for $1\leq k \leq a$.}
\item{$c(a+1,b,0)=c(a,b,0)q^b$.}
\item{$c(a+1,b,a+1)=c(a,b,a)[b+a].$}
\end{enumerate}}
\end{prop}
\begin{proof}
By Lemma \ref{p1} and Definition \ref{nc} we have
$$yx^b=\sum_{k=0}^1c(1,b,k) x^{b+k}y^{1-k}=q^bx^by + [b]x^{b+1},$$
thus $c(1,b,0)=q^b$ and $c(1,b,1)=[b]$. On the other hand we
compute
\begin{eqnarray*}
y^{a+1}x^b&=& \sum_{k=0}^a c(a,b,k)(yx^{b+k})y^{a-k} \\
\mbox{}&\mbox{}&\mbox{}\\
\mbox{}&=& \sum_{k=0}^a c(a,b,k)(q^{b+k}x^{b+k}y + [b+k]x^{b+k+1}) y^{a-k}\\
\mbox{}&=&c(a,b,0)q^bx^by^{a+1}+ \sum_{k=1}^a
c(a,b,k)q^{b+k}x^{b+k}y^{a+1-k}\\
\mbox{}&\mbox{}&  + \sum_{k=1}^a
c(a,b,k-1)[b+k-1]x^{b+k}y^{a+1-k}+c(a,b,a)[b+a]x^{a+b+1}.
\end{eqnarray*}

By definition we have that
$$y^{a+1}x^b=\sum_{k=0}^{a+1} c(a+1,b,k)x^{b+k}y^{a+1-k}.$$
Therefore we have shown that

\begin{eqnarray*}
\sum_{k=0}^{a+1}
c(a+1,b,k)x^{b+k}y^{a+1-k}&=& c(a,b,0)q^bx^by^{a+1}\\
&+& \sum_{k=1}^a \left(c(a,b,k)q^{b+k} + c(a,b,k-1)[b+k-1] \right)x^{b+k}y^{a+1-k}\\
&+& c(a,b,a)[b+a]x^{a+b+1}.
\end{eqnarray*}
Considering this equality termwise gives the desired identities.

\end{proof}

Notice that the first identity from Proposition \ref{ayuda} together with
the initial conditions $c(0,b,k)= \delta_{0,k}$ completely determine the function
$c(a,b,k).$  We shall use this fact in the proof of Theorem \ref{sm}. Our next result shows that  $c(a,b,a)$ is the $q$-analogue
of the increasing factorial.

\begin{lema}{\em
\begin{enumerate}
\item{$c(a,b,0)=q^{ab}$.}
\item{$c(a,b,a)=[b][b+1]\dots[b+a-1]=[b]^{(a)}$.}
\end{enumerate}}
\end{lema}

\begin{proof}
Clearly $c(1,b,0)=q^{b}.$ Moreover by Proposition
\ref{ayuda} we have that
$$c(a+1,b,0)=c(a,b,0)q^b=q^{ab}q^b=q^{(a+1)b}.$$
For $a=1$ we have $c(1,b,1)=[b]^{(1)}=[b],$ and using again Proposition
\ref{ayuda} we get
$$c(a+1,b,a+1)=c(a,b,a)[b+a]=[b]^{(a)}[b+a]=[b]^{(a+1)}.$$
\end{proof}

We are ready to discuss the combinatorial interpretation of the
normal polynomials $c(a,b,k)$. Let $P_k[[1,a]]$ be the set of
subsets of $[[1,a]]$ with $k$ elements. We define a $q$-weight
$$\omega_b:P_k[[1,a]]
\longrightarrow \mathbb{N}[q]$$ which sends $A \in P_k[[1,a]]$
into
$${\displaystyle\omega_b(A)= [b]^{(k)}q^{(a-k)b}q^{\sum_{i\in
A^{c}}|A_{<i}|}.}$$

\begin{thm}\label{sm}{\em For $(a,b) \in \mathbb{N}\times \mathbb{N}_+$  and $0 \leq k \leq
a,$ we have that $c(a,b,k)=|P_k[[1,a]], \omega_b |$.}
\end{thm}

\begin{proof}
We have to show that
$${\displaystyle c(a,b,k)=|P_k(a), \omega_b|=[b]^{(k)}q^{(a-k)b}\sum_{A\in P_k[[1,a]]} q^{ \sum_{i\in
A^{c}}|A_{<i}|}}.$$

Let $\overline{c}(a,b,k)$ be given by the
right hand side of formula above for $a \geq 1$ and
$\overline{c}(0,b,k) = \delta_{0, k}$. We must show that
$\overline{c}(a,b,k)= c(a,b,k).$ Since $\overline{c}(0,b,k)=
c(0,b,k),$ it is enough to show that $\overline{c}(a,b,k)$
satisfies, for $1\leq k
\leq a,$ the recursion
$$\overline{c}(a+1,b,k)=\overline{c}(a,b,k)q^{b+k} + \overline{c}(a,b,k-1)[b+k-1].$$
Sets $A \in P_k[[1, a+1]]$ are classified in two blocks according
to whether $ a+1
\notin A$ or $a+1 \in A.$ Thus we obtain that
$${\displaystyle \overline{c}(a+1,b,k)=|P_k(a+1), \omega_b|=[b]^{(k)}q^{(a-k+1)b}\sum_{A\in P_k[[1,a+1]]} q^{ \sum_{i\in
A^{c}}|A_{<i}|}}$$ is equal to the sum of two terms
$$\left( [b]^{(k)}q^{(a-k)b}\sum_{A\in P_k[[1,a]]} q^{ \sum_{i\in
A^{c}}|A_{<i}|} \right)q^{b+k} \ \ +$$
$$\left( [b]^{(k-1)}q^{(a-k+1)b}\sum_{A\in P_{k-1}[[1,a]]} q^{ \sum_{i\in
A^{c}}|A_{<i}|} \right)[b + k -1].$$ Thus the numbers
$\overline{c}(a,b,k)$ satisfy the required recursion.
\end{proof}

Let us remark that writing $A \in P_k[[1,a]]$ as $A=\{t_1 < t_2 <
\cdots < t_k  \}$, using the elementary identity
$$\sum_{i\in A^{c}}|A_{<i}|=\sum_{s=1}^k s (t_{s+1} - t_s  -1)$$
and setting $t_{k+1}=a+1$ we obtain that:
$$c(a,b,k)=[b-1]^{(k)}q^{(a-k)b} \sum_{1 \leq t_1 <\cdots < t_k \leq a}
q^{\sum_{s=1}^k s (t_{s+1} - t_s  -1)} .$$

\section{Normal polynomials and symmetric powers of $MW_q$}

In this section we find explicit formulae for the normal
polynomials of the algebra $MW_q$. We also begin the study of the
symmetric power of that algebra.

\begin{defi}\label{mw}{\em Let $A=(A_1,\cdots,A_n)\in
(\mathbb{N}^{2})^{n}$ with $A_i=(a_i,b_i)$. The normal polynomial
$ N(A,k,q) \in \mathbb{N}[q]$ is defined by the following identity
in $MW_q$: $$\prod_{i=1}^nx^{a_i}y^{b_i}=
\sum_{k=0}^{|b|} N(A,k,q)x^{|a|+k}y^{|b|-k}.$$ For $k>|b|$ we set $N(A,k,q)=0$.}
\end{defi}

Recall from Section 2 that the notation $p \vdash k$ means that $p$ is a vector
$(p_1,\cdots,p_{n-1}) \in \mathbb{N}^{n-1}$ such
that ${\displaystyle |p|=\sum_{i=1}^{n-1}p_i=k}$.
Our next result is obtained using  Definition \ref{nc} several times.

\begin{thm}\label{juju}{\em For $(A,k) \in (\mathbb{N}^{2})^{n}
\times \mathbb{N}$ we have that $$
N(A,k,q)=\sum_{p
\vdash k}\prod_{i=1}^{n-1} c( b_i,|a_{>i}|+|p_{>i}|,p_i),$$ where the partition $p$ of $k$ must be
such that
$0
\leq p_i\leq b_i$ for $i \in [[1, n-1]]$.}
\end{thm}

It is not hard to show that the normal polynomial may also be
computed via the identity $$N(A,k,q)=\sum_{p
\vdash k}\prod_{i=1}^{n-1} c(|b_{\leq
i}|-|p_{<i}|,a_{i+1},p_i),$$ where $0 \leq p_i\leq |b_{\leq
i}|-|p_{< i}|$ for $i \in [[1, n-1]]$.\\

Applying Theorem $\ref{juju}$, specialized in the representation
$\rho$, to $x^{-t}$ we obtain that if $(a,b,t)\in
\mathbb{N}^{n}\times \mathbb{N}^{n} \times \mathbb{N}_+$ then
$$\prod_{i=1}^n [t+|b_{\geq i}|+|a_{>i}|-1]=
\sum_{k=0}^{|b|}\left(\sum_{p
\vdash k}\prod_{i=1}^{n-1} c(
b_i,|a_{>i}|+|p_{>i}|,p_i)\right)[t+|b|-k-1],$$ where $0 \leq
p_i\leq b_i$ for $i\in [[1,n-1]]$.\\

Using the alternative expression for $N(A,k,q)$ given above, one
obtains that:
$$\prod_{i=1}^n [t+|b_{\geq i}|+|a_{>i}|-1]=
\sum_{k=0}^{|b|}\left(\sum_{p
\vdash k}\prod_{i=1}^{n-1} c(|b_{\leq
i}|-|p_{<i}|,a_{i+1},p_i)\right)[t+|b|-k-1],$$
 where $0 \leq p_i\leq |b_{\leq
i}|-|p_{< i}|$ for $i \in [[1, n-1]]$.\\

If instead of $\rho$ we use the representation $\iota$ applied to
$x^{t}$ we get the identity:
$$\frac{1}{\prod_{i=1}^{n}[t + |a_{{\geq i}}| + |b_{\geq i}| + 1]^{(a_i)} }= \sum_{k=0}^{|b|}
\left(\sum_{p
\vdash k}\prod_{i=1}^{n-1} c( b_i,|a_{>i}|+|p_{>i}|,p_i)\right)
\frac{1}{[t+|a|+|b|]^{(|a|+k)}},$$ where $0 \leq p_i\leq b_i$ for
$i \in [[1, n-1]]$.\\

Also with the alternative expression for $N(A,k,q)$  we get:
$$\frac{1}{\prod_{i=1}^{n}[t + |a_{{\geq i}}| + |b_{\geq i}| + 1]^{(a_i)} }= \sum_{k=0}^{|b|}
\left(\sum_{p \vdash k}\prod_{i=1}^{n-1} c(|b_{\leq i}|-|p_{<i}|,a_{i+1},p_i)\right)
\frac{1}{[t+|a|+|b|]^{(|a|+k)}},$$ where $0 \leq p_i\leq |b_{\leq
i}|-|p_{< i}|$ for $i \in [[1, n-1]]$.\\

Next we provide explicit formulae for the products of several
elements in the $n$-th symmetric power $\sym^{n}(MW_q)$ of the
$q$-meromorphic Weyl algebra $MW_q$.

\begin{thm}{\em
For each map $(a,b):[[1,m]]\!\times\! [[1,n]]\longrightarrow \mathbb{N}^{2}$
the following identity holds in $\sym^{n}(MW):$
\begin{equation*}
{\des (n!)^{m-1}\prod_{i=1}^{m}\overline{\prod_{j=1}^{n}
x_j^{a_{ij}}y_j^{b_{ij}}}=\sum_{\sigma,k,p}\left(\prod_{l=1}^{m-1}\prod_{j=1}^{n} c((b_j^{\sigma})_l,|(a_j^{\sigma})_{>l}
|+|p^j_{>l}|,p_{l}^j)\right)
\overline{\prod_{j=1}^{n} x_j^{|a_j^{\sigma}| + k_j}y_j^{|b_j^{\sigma}|-k_j}}.}
\end{equation*}

In the formula above we are using the following conventions:
$\sigma\in {\{1 \} \times S_n^{m-1}}$, $k\in\mathbb{N}^{n}$ is such that
$k_j \leq |b_j^{\sigma}|$, $p=(p^1,...,p^{n})\in (\mathbb{N}^{m-1})^{n}$, $p^j=( p^j_1,...,p_{m-1}^j )$,
$a_j^{\sigma}=(a_{1\sigma_1^{-1}(j)},..., a_{m\sigma_m^{-1}(j)}),$ and
$b_j^{\sigma}=(b_{1\sigma_1^{-1}(j)},..., b_{m\sigma_m^{-1}(j)})$

}
\end{thm}

The explicit computation of products in $\sym^{n}(MW_q)$ is rather
difficult as the following example shows.
\begin{exa}{\em
For $n=2, m=2$ we have
$$2(x_1y_1x_2^2y_2)(x_1^2y_1^{2}x_2y_2)=x_1y_1x_2^2y_2x_1^2y_1^2x_2y_2+x_1y_1x_2^2y_2x_1y_1x_2^2y_2^2=$$
$$q^{3}x_1^{3}y_1^{3}x_2^{3}y_2^{2}+q^{2}x_1^{3}y_1^{3}x_2^{4}y_2+(q^{2}+q)x_1^{4}y_1^{2}x_2^{3}y_2^{2}
+(q+1)x_1^{4}y_1^{2}x_2^{4}y_2+$$
$$q^{3}x_1^{2}y_1^{2}x_2^{4}y_2^{3}+(q^{2}+q)x_1^{2}y_1^{2}x_2^{5}y_2^2+q^{2}x_1^{3}y_1x_2^{4}y_2^{3}
+(q+1)x_1^{3}y_1x_2^{5}y_2^2.$$}
\end{exa}

We close this work mentioning a couple of research problems. First, it would be interesting to study the Hochschild
cohomology of the meromorphic and $q$-meromorphic Weyl algebras and
their corresponding symmetric powers along the lines developed in
\cite{Alev1, Alev2}. Second, using techniques introduced in
\cite{RE} and further developed in \cite{Blan1, Blan12, Blan2,
cas1} we have constructed a categorification of the  Weyl algebra,
and more generally of the Kontsevich  star product
\cite{K} for Poisson structures on $\mathbb{R}^n.$ It would be
interesting to study the categorification of the meromorphic and
$q$-meromorphic Weyl algebras.

\subsection*{Acknowledgements}

Our thanks to Nicolas Andruskiewitsch, Takashi Kimura and Sylvie
Paycha. We also thank a couple of anonymous
referees for helpful suggestions and comments. The second author thanks the organizing committee for
inviting her to participate in the  "XVII Coloquio Latinoamericano
de Algebra" held in Medell\'in, Colombia, 2007.

\bigskip

\bigskip

\noindent ragadiaz@gmail.com\\
\noindent Facultad de Administraci\'on, Universidad
del Rosario, Bogot\'a, Colombia \\

\noindent epariguan@javeriana.edu.co \\
Departamento de Matem\'aticas, Pontificia Universidad Javeriana,
Bogot\'a, Colombia

\end{document}